\documentclass[12pt,a4paper,reqno]{amsart}
\usepackage{amsmath}
\usepackage{amsthm}
\usepackage{amsfonts}
\usepackage{amssymb}
\usepackage{color}
\usepackage{mathrsfs}
\usepackage{graphicx}

\numberwithin{equation}{section}

\theoremstyle{plain}
\newtheorem{theorem}{Theorem}

\theoremstyle{plain}
\newtheorem{question}{Question}

\theoremstyle{plain}
\newtheorem{lemma}{Lemma}[section]

\theoremstyle{remark}
\newtheorem*{remark}{Remark}

\def\bfs{\mathbf{s}}

\def\bfv{\mathbf{v}}
\def\bfw{\mathbf{w}}
\def\bfx{\mathbf{x}}
\def\bfy{\mathbf{y}}

\def\T{\mathbf{T}}

\def\dd{\mathrm{d}}

\def\eps{\varepsilon}

\def\Qq{\mathbb{Q}}
\def\Rr{\mathbb{R}}

\def\BBB{\mathcal{B}}

\def\EEE{\mathcal{E}}

\def\LLL{\mathcal{L}}
\def\MMM{\mathcal{M}}

\def\PPP{\mathcal{P}}

\def\SSS{\mathcal{S}}

\def\WWW{\mathcal{W}}

\def\YYY{\mathcal{Y}}

\def\NNNN{\mathscr{N}}

\def\SSSS{\mathscr{S}}

\renewcommand{\le}{\leqslant}
\renewcommand{\ge}{\geqslant}

\title[Kronecker--Weyl equidistribution theorem]
{A note on the Kronecker--Weyl\\
equidistribution theorem}

\author[Beck]{J. Beck}

\address{Department of Mathematics, Hill Center for the Mathematical Sciences, Rutgers University, Piscataway NJ 08854, USA}

\email{jbeck@math.rutgers.edu}

\author[Chen]{W.W.L. Chen}

\address{School of Mathematical and Physical Sciences, Faculty of Science and Engineering, Macquarie University, Sydney NSW 2109, Australia}

\email{william.chen@mq.edu.au}

\author[Yang]{Y. Yang}

\address{School of Science, Beijing University of Posts and Telecommunications, Beijing 100876, China}

\email{yangyx@bupt.edu.cn}

\begin{document}



\begin{abstract}
We study the relationship between the discrete and the continuous versions of the
Kronecker--Weyl equidistribution theorem, as well as their possible extension to manifolds in higher dimensions.
We also investigate a way to deduce in some limited way uniformity results in higher dimension from results in lower dimension.
\end{abstract}

\maketitle

\thispagestyle{empty}

%
%

\section{Introduction}\label{sec1}

There are two versions of the classical Kronecker--Weyl equidistribution theorem.
Let $d$ be a fixed positive integer.
If $v_1,\ldots,v_d,1$ are linearly independent over~$\Qq$,
then the vector $\bfv=(v_1,\ldots,v_d)\in\Rr^d$ is called a Kronecker vector,
and the vector $\bfv^*=(v_1,\ldots,v_d,1)\in\Rr^{d+1}$ is called a Kronecker direction.
The continuous version concerns the distribution of half-infinite geodesics
with Kronecker directions in the unit torus $[0,1)^{d+1}$,
and much of this monograph concerns extensions of this version from the unit torus $[0,1)^2$
to arbitrary finite polysquare translation surfaces.

On the other hand, the discrete version of the Kronecker--Weyl equidistribution theorem
concerns the distribution of half-infinite Kronecker sequences in the unit torus $[0,1)^d$.
A natural first question is then to attempt to extend this version from the unit torus $[0,1)^2$
to arbitrary finite polysquare translation surfaces.

In general, for any fixed positive integer~$d$, we consider the continuous problem concerning
the distribution of half-infinite geodesics with Kronecker directions in finite polycube translation manifolds in $d+1$ dimensions,
as well as the discrete problem of the distribution of half-infinite Kronecker sequences
in finite polycube translation manifolds in $d$ dimensions.
The following are natural questions:

\begin{question}\label{q1}
Is it true that, for any integer $d\ge1$, any half-infinite geodesic with a Kronecker direction
in a finite polycube translation manifold in $d+1$ dimensions is uniformly distributed?
If not, then under what condition can we guarantee uniform distribution?
\end{question}

\begin{question}\label{q2}
Is it true that, for any integer $d\ge2$, any half-infinite Kronecker sequence
in a finite polycube translation manifold in $d$ dimensions is uniformly distributed?
If not, then under what condition can we guarantee uniform distribution?
\end{question}

\begin{question}\label{q3}
The classical Kronecker--Weyl equidistribution theorem on the unit torus has some time-quantitative extensions
with explicit error terms.
Under what conditions can we establish time-quantitative uniformity in these more general settings?
\end{question}

For Question~\ref{q1}, the Gutkin--Veech theorem \cite{gutkin84,veech87} answers the special case $d=1$ in the affirmative.
However, for $d=2$, we are currently not able to establish uniformity results for half-infinite geodesics
in a general finite polycube translation $3$-manifold.

In this paper, we study the relationship between the discrete and the continuous versions
of possible non-integrable analogues of the Kronecker--Weyl equidistribution theorem
concerning finite polysquare translation surfaces and some related finite polycube translation $3$-manifolds.
In Section~\ref{sec2}, we investigate a relationship between the discrete version and the continuous version in the special case $d=2$.
Theorem~\ref{thm2}, a by-product of this study, gives an affirmative answer to the special case $d=2$ of Question~2.
Then in Section~\ref{sec3}, we develop a way to step up the problem by one dimension, and this leads to various infinite classes
of polycube translation manifolds where half-infinite Kronecker sequences and half-infinite geodesics with Kronecker directions
are uniformly distributed.

%
%

\section{A simple equivalence principle}\label{sec2}

Before we go any further, it is appropriate that we understand what we mean by a Kronecker sequence
on a finite polysquare translation surface~$\PPP$.
To properly define a half-infinite Kronecker sequence $\bfv_0+j\bfv$, $j=0,1,2,3,\ldots,$
with starting point $\bfv_0$ and step vector $\bfv$ on a polysquare translation surface~$\PPP$,
we need a supporting half-infinite geodesic $\LLL(t)$, $t\ge0$, with $\LLL(0)=\bfv_0$ on $\PPP$
and a time step $g\in\Rr$, such that the finite geodesic segment $\LLL(t)$, $0\le t\le g$,
is in the direction of the step vector $\bfv$ and has length equal to~$\vert\bfv\vert$.
Then $\bfv_0+j\bfv=\LLL(jg)$, $j=0,1,2,3,\ldots.$
Clearly the Kronecker sequence is half-infinite if and only if $\bfv_0$ is a non-pathological
starting point of a geodesic with direction $\bfv$ on~$\PPP$.

Suppose that $\PPP$ is a polysquare translation surface with $s$ atomic squares,
and that $\MMM=\PPP\times[0,1)$ denotes the associated polycube translation $3$-manifold
which is the cartesian product of $\PPP$ and the unit torus $[0,1)$.

A vector $\bfv=(v_1,v_2)\in\Rr^2$ is said to be a Kronecker vector
if $v_1,v_2,1$ are linearly independent over~$\Qq$.
We are interested in the distribution of a Kronecker sequence
$\bfv_0+j\bfv$, $j=0,1,2,3,\ldots,$ with starting point $\bfv_0$
on the polysquare translation surface~$\PPP$.

A vector $\bfv^*=(v_1,v_2,1)\in\Rr^3$ is said to be a Kronecker direction
if $v_1,v_2,1$ are linearly independent over~$\Qq$.
We are interested in the distribution of a half-infinite geodesic $\LLL(t)$, $t\ge0$,
with starting point $\LLL(0)$ and direction $\bfv^*$ in the associated polycube translation $3$-manifold~$\MMM$.

We have a simple equivalence principle relating half-infinite Kronecker sequences and half-infinite geodesics
with Kronecker direction.

\begin{theorem}\label{thm1}
Suppose that $\PPP$ is a finite polysquare translation surface,
and that $\MMM=\PPP\times[0,1)$ denotes the associated polycube translation $3$-manifold
which is the cartesian product of $\PPP$ and the unit torus $[0,1)$.
Suppose also that $\bfv=(v_1,v_2)\in\Rr^2$ is a Kronecker vector, so that $\bfv^*=(v_1,v_2,1)\in\Rr^3$
is a Kronecker direction.
Then the following two statements are equivalent:

\emph{(i)}
Every half-infinite Kronecker sequence $\bfv_0+j\bfv$, $j=0,1,2,3,\ldots,$ on $\PPP$
is uniformly distributed.

\emph{(ii)}
Every half-infinite geodesic $\LLL(t)$, $t\ge0$, with direction $\bfv^*$ in $\MMM$
is uniformly distributed.
\end{theorem}

\begin{proof}[Sketch of proof]
((ii) $\Rightarrow$ (i))
The argument here is rather simple.
Suppose that $\PPP$ has $s$ atomic squares.
To establish (i), let $S$ be a convex set in an atomic square of~$\PPP$.
Consider the first $J$ terms of the Kronecker sequence, given by
$\bfv_0+j\bfv$, $j=0,1,\ldots,J-1$.
The number of terms of this finite sequence in $S$ is given by
\begin{displaymath}
\vert\{j=0,1,\ldots,J-1:\bfv_0+j\bfv\in S\}\vert,
\end{displaymath}
with corresponding expectation given by
\begin{displaymath}
\frac{J\lambda_2(S)}{s},
\end{displaymath}
where $\lambda_2$ denotes $2$-dimensional Lebesgue measure.
To establish uniformity of the Kronecker sequence on~$\PPP$, we need to show that
\begin{equation}\label{eq2.1}
\vert\{j=0,1,\ldots,J-1:\bfv_0+j\bfv\in S\}\vert\bigg/\frac{J\lambda_2(S)}{s}\to1
\quad\mbox{as $J\to\infty$}.
\end{equation}
Let $\LLL(t)$, $t\ge0$, be a half-infinite geodesic with starting point $\LLL(0)=(\bfv_0,0)$
and direction $\bfv^*$ on~$\MMM$.
Let $S^*\subset\MMM$ be obtained by sweeping the set $S\times\{0\}$ along by the vector~$\bfv^*$,
so that $\lambda_3(S^*)=\lambda_2(S)$,
where $\lambda_3$ denotes $3$-dimensional Lebesgue measure.
Consider a finite geodesic segment $\LLL(t)$, $0\le t\le T=J(v_1^2+v_2^2+1)^{1/2}$.
Then the total length of the parts of this geodesic segment in $S^*$ is given by
\begin{align}\label{eq2.2}
&
\vert\{0\le t\le T:\LLL(t)\in S^*\}\vert
\nonumber
\\
&\quad
=(v_1^2+v_2^2+1)^{1/2}\vert\{j=0,1,\ldots,J-1:\bfv_0+j\bfv\in S\}\vert,
\end{align}
with corresponding expectation given by
\begin{equation}\label{eq2.3}
\frac{T\lambda_3(S^*)}{s}=\frac{J(v_1^2+v_2^2+1)^{1/2}\lambda_2(S)}{s}.
\end{equation}
The set $S^*$ is a union of finitely many convex sets in atomic cubes of~$\MMM$.
Suppose that (ii) holds.
Uniformity of the half-infinite geodesic in $\MMM$ then implies that
\begin{equation}\label{eq2.4}
\vert\{0\le t\le T:\LLL(t)\in S^*\}\vert\bigg/\frac{T\lambda_3(S^*)}{s}\to1
\quad\mbox{as $T\to\infty$}.
\end{equation}
It is clear that \eqref{eq2.1} follows immediately on combining \eqref{eq2.2}--\eqref{eq2.4}.

((i) $\Rightarrow$ (ii))
In \cite[Section~3.4.1]{BDY20a} or \cite[Section~3.6]{BCY24}, a result concerning uniformity
of a half-infinite geodesic is deduced from a result concerning uniformity of a sequence,
by an application of the Koksma inequality.
Here we apply an analogous Koksma inequality type argument.
\end{proof}

The Gutkin--Veech theorem gives uniformity to any half-infinite geodesic with Kronecker direction on a finite polysquare translation surface.
As a simple application of Theorem~\ref{thm1}, we establish the analogous result for half-infinite Kronecker sequences.

\begin{theorem}\label{thm2}
Any half-infinite Kronecker sequence $\bfv_0+j\bfv$, $j=0,1,2,3,\ldots,$ on a finite polysquare translation surface $\PPP$
is uniformly distributed.
\end{theorem}

\begin{proof}
Let $\MMM=\PPP\times[0,1)$ denote the associated polycube translation $3$-manifold.
In view of Theorem~\ref{thm1}, it suffices to show that
every half-infinite geodesic $\LLL(t)$, $t\ge0$, with Kronecker direction in $\MMM$ is uniformly distributed.
This latter condition is the conclusion of Theorem~\ref{thm3}.
\end{proof}

The following result is \cite[Theorem~3]{BCY-a}.

\begin{theorem}\label{thm3}
Suppose that a polycube translation $3$-manifold $\MMM$ is the cartesian product of a finite polysquare translation surface $\PPP$
and the unit torus $[0,1)$.
Then any half-infinite geodesic in $\MMM$ with a Kronecker direction $\bfv^*\in\Rr^3$ is uniformly distributed unless it hits a singularity.
\end{theorem}

%
%

\section{Stepping up principle}\label{sec3}

Theorem~\ref{thm2}, concerning the distribution of half-infinite Kronecker sequences on~$\PPP$, can be interpreted
as a step-up from the Gutkin--Veech theorem, concerning the distribution of half-infinite geodesics with Kronecker direction on~$\PPP$.
In view of the equivalence given by Theorem~\ref{thm1}, the stepping-up result is Theorem~\ref{thm3},
which establishes the uniformity of a half-infinite geodesic with Kronecker direction on the associated polycube translation $3$-manifold $\MMM$
from the uniformity of a half-infinite geodesic with Kronceker direction on the finite polysquare translation surface~$\PPP$.
This represents a step-up in dimension in some restricted way.

In this section, we expand on this idea.
We establish the following results which, for simplicity, we state only in $2$ and $3$ dimensions.
There are analogues in $d$ and $d+1$ dimensions for every integer $d\ge3$.

The discrete versions of these results concern the distribution of half-infinite Kronecker sequences
on a finite polysquare translation surface $\PPP$ and the analogous problem
in the associated polycube translation $3$-manifold $\MMM$
which is the cartesian product of $\PPP$ and the unit torus $[0,1)$.

\begin{theorem}\label{thm4}
Suppose that $\PPP$ is a finite polysquare translation surface,
and that $\MMM=\PPP\times[0,1)$ denotes the associated polycube translation $3$-manifold
which is the cartesian product of $\PPP$ and the unit torus $[0,1)$.
Suppose also that $\bfv\in\Rr^2$ is the step vector of a Kronecker sequence on~$\PPP$.
Then the following statements are equivalent:

\emph{(i)}
Every half-infinite Kronecker sequence with step vector $\bfv$ on $\PPP$ is uniformly distributed.

\emph{(ii)}
For any $w\in\Rr$ such that $\bfw=(\bfv,w)\in\Rr^3$ is the step vector of a Kronecker sequence on~$\MMM$,
every half-infinite Kronecker sequence with step vector $\bfw$ in $\MMM$ is uniformly distributed.
\end{theorem}

\begin{theorem}\label{thm5}
Under the hypotheses of Theorem~\ref{thm4}, the following statements are equivalent:

\emph{(i)}
The $\bfv$-shift on $\PPP$ is ergodic.

\emph{(ii)}
For any $w\in\Rr$ such that $\bfw=(\bfv,w)\in\Rr^3$ is the step vector of a Kronecker sequence on~$\MMM$,
the $\bfw$-shift in $\MMM$ is ergodic.
\end{theorem}

The continuous versions of these results concern the distribution of half-infinite geodesics with Kronecker direction
on a finite polysquare translation surface $\PPP$ and the analogous problem
in the associated polycube translation $3$-manifold~$\MMM$.

\begin{theorem}\label{thm6}
Suppose that $\PPP$ is a finite polysquare translation surface,
and that $\MMM=\PPP\times[0,1)$ denotes the associated polycube translation $3$-manifold
which is the cartesian product of $\PPP$ and the unit torus $[0,1)$.
Suppose also that $\bfv\in\Rr^2$ is a Kronecker direction.
Then the following statements are equivalent:

\emph{(i)}
Every half-infinite geodesic with direction $\bfv$ on $\PPP$ is uniformly distributed.

\emph{(ii)}
For any $w\in\Rr$ such that $\bfw=(\bfv,w)\in\Rr^3$ is a Kronecker direction,
every half-infinite geodesic with direction $\bfw$ in $\MMM$ is uniformly distributed.
\end{theorem}

\begin{theorem}\label{thm7}
Under the hypotheses of Theorem~\ref{thm4}, the following statements are equivalent:

\emph{(i)}
Geodesic flow with direction $\bfv$ on $\PPP$ is ergodic.

\emph{(ii)}
For any $w\in\Rr$ such that $\bfw=(\bfv,w)\in\Rr^3$ is a Kronecker direction,
geodesic flow with direction $\bfw$ in $\MMM$ is ergodic.
\end{theorem}

\begin{remark}
For Theorems \ref{thm6} and~\ref{thm7}, we already know that (i) and (ii) hold,
in view of the Gutkin--Veech theorem and Theorem~\ref{thm3},
so only the analogues in $d$ and $d+1$ dimensions for integers $d\ge3$ are new.
\end{remark}

We concentrate our efforts on establishing Theorem~\ref{thm4}.
That (ii) implies (i) is almost trivial, by projection from $\MMM$ to~$\PPP$.
The converse is considerably harder.

Let $\bfw=(v_1,v_2,w)\in\Rr^3$, where $v_1,v_2,w,1$ are linearly independent over~$\Qq$,
and consider the $\bfw$-shift $\T^*_\bfw=\T^*_\bfw(\MMM)$ of geodesic flow in direction $\bfw$ in $\MMM=\PPP\times[0,1)$.
We can consider two projections of~$\T^*_\bfw$.

On the one hand, we can project the $\bfw$-shift $\T^*_\bfw$ to the unit torus $[0,1)^3$,
simply by taking every coordinate modulo~$1$,
leading to the $\bfw$-shift $\T_\bfw=\T_\bfw([0,1)^3)$ in the unit torus $[0,1)^3$
which is ergodic.

On the other hand, we can project the $\bfw$-shift $\T^*_\bfw$ to the polysquare translation surface~$\PPP$,
simply by dropping reference to the $3$-rd coordinates throughout,
leading to the $\bfv$-shift $\T^*_\bfv=\T^*_\bfv(\PPP)$ on~$\PPP$.

\begin{lemma}\label{lem31}
Suppose that the condition (i) in Theorem~\ref{thm4} holds.
Then the $\bfv$-shift $\T^*_\bfv=\T^*_\bfv(\PPP)$ on $\PPP$ is ergodic.
\end{lemma}

\begin{proof}[Sketch of proof]
If a Kronecker sequence on $\PPP$ becomes undefined
after finitely many terms, then the supporting geodesic $\LLL(t)$, $t\ge0$, with starting point $\LLL(0)=\bfv_0$
and direction $\bfv$ on $\PPP$ hits a singular point of~$\PPP$.
The collection of singular points of $\PPP$ clearly has $2$-dimensional Lebesgue measure~$0$.
Thus the collection of starting points $\bfv_0$ that lead a Kronecker sequence to be undefined
after finitely many terms also has $2$-dimensional Lebesgue measure~$0$.
Thus almost every point $\bfv_0\in\PPP$ gives rise to a half-infinite Kronecker sequence.
Thus the condition (i) in Theorem~\ref{thm4} guarantees that for almost every starting point~$\bfv_0$,
the half-infinite Kronecker sequence $\bfv_0+j\bfv$, $j=0,1,2,3,\ldots,$ on $\PPP$ is uniformly distributed.
Hence the visiting density of the sequence in any Jordan measurable set $A$ is given by $\lambda_2(A)/s$.

Suppose, on the contrary, that the $\bfv$-shift $\T^*_\bfv$ on $\PPP$ is not ergodic.
Then there is a partition $\PPP=U_1\cup U_2$, where the subsets $U_1,U_2\subset\PPP$ are $\T^*_\bfv$-invariant,
and satisfy $\lambda_2(U_1)>0$ and $\lambda_2(U_2)>0$.
Moreover, since the projection of $\T^*_\bfv$ to the unit torus $[0,1)^2$, simply by taking every coordinate modulo~$1$,
leads to the $\bfv$-shift $\T_\bfv=\T_\bfv([0,1)^2)$ on the unit torus $[0,1)^2$ which is ergodic,
there is a decomposition $\PPP=M_1\cup\ldots\cup M_k$ for some integer $k$ satisfying $2\le k\le s$,
where $s$ is the number of atomic squares in $\PPP$, such that for every $i=1,\ldots,k$,
the set $M_i$ is $\T^*_\bfv$-invariant and does not contain a proper $\T^*_\bfv$-invariant subset,
and $\lambda_2(M_i)$ is a positive integer.
We say that the subsets $M_1,\ldots,M_k$ are \textit{minimal}.

We can find a Jordan measurable subset $A\subset\PPP$ such that
\begin{displaymath}
\lambda_2(A\cap M_1)<\frac{1}{10}
\quad\mbox{and}\quad
\lambda_2(A\cap M_2)>\frac{\lambda_2(M_2)}{2}.
\end{displaymath}
Next we apply the Birkhoff ergodic theorem to both $M_1$ and~$M_2$.
The restriction of the $\bfv$-shift $\T_\bfv^*$ to $M_1$ and to $M_2$ are both ergodic.
In each case, the simplest form of the ergodic theorem says that the \textit{time average}
is equal to the \textit{space average}.
Then for almost every starting point $\bfv_0\in M_1$, the visiting density of the Kronecker sequence
$\bfv_0+j\bfv$, $j=0,1,2,3,\ldots,$ on $\PPP$ in the set $A$ is equal to the relative measure
\begin{displaymath}
\frac{\lambda_2(A\cap M_1)}{\lambda_2(M_1)}<\frac{1}{10},
\end{displaymath}
while for almost every starting point $\bfv_0\in M_2$, the visiting density of the Kronecker sequence
$\bfv_0+j\bfv$, $j=0,1,2,3,\ldots,$ on $\PPP$ in the set $A$ is equal to the relative measure
\begin{displaymath}
\frac{\lambda_2(A\cap M_2)}{\lambda_2(M_2)}>\frac{1}{2}.
\end{displaymath}
Thus at least one of these is different from $\lambda_2(A)/s$, leading to a contradiction.
\end{proof}

\begin{proof}[Proof of Theorem~\ref{thm4}]
((i) $\Rightarrow$ (ii))
We need to prove that the $\bfw$-shift $\T^*_\bfw$ in $\MMM$ is ergodic.
Suppose on the contrary that this is not the case.
Then there exists a non-trivial partition $\MMM=\WWW\cup\SSS$ into a union of two $\T^*_\bfw$-invariant subsets
$\WWW,\SSS\subset\MMM$, called White and Silver, say, each with integral measure and such that
\begin{displaymath}
1\le\lambda_3(\WWW),\lambda_3(\SSS)\le s-1,
\end{displaymath}
where $s$ is the number of atomic cubes in $\MMM$ and $\lambda_3$ denotes $3$-dimensional Lebesgue measure.

Let $\YYY_1,\ldots,\YYY_s$ denote the atomic cubes of~$\MMM$,
and consider the projection of $\MMM$ to the unit torus $[0,1)^3$.
Then for any point $P\in[0,1)^3$, there are precisely $s$ points
\begin{displaymath}
P_1\in\YYY_1,
\quad
\ldots,
\quad
P_s\in\YYY_s
\end{displaymath}
that have the same image $P$ under this projection.
Let
\begin{displaymath}
f_\WWW(P)=\vert\{P_1,\ldots,P_s\}\cap\WWW\vert
\quad\mbox{and}\quad
f_\SSS(P)=\vert\{P_1,\ldots,P_s\}\cap\SSS\vert.
\end{displaymath}
Then $f_\WWW$ and $f_\SSS$ are positive integer valued functions defined on $[0,1)^3$ such that
\begin{displaymath}
f_\WWW(P)+f_\SSS(P)=s
\quad\mbox{for almost every $P\in[0,1)^3$}.
\end{displaymath}
The $\bfw$-shift $\T_\bfw$ in the unit torus $[0,1)^3$ is ergodic.
Since the functions $f_\WWW$ and $f_\SSS$ are $\T_\bfw$-invariant,
it follows from the Birkhoff ergodic theorem that they are constant almost everywhere in $[0,1)^3$.
Then
\begin{equation}\label{eq3.1}
f_\WWW+f_\SSS=s
\quad\mbox{and}\quad
1\le f_\WWW,f_\SSS\le s-1.
\end{equation}

Let $\chi_\WWW$ denote the characteristic function of $\WWW$ in~$\MMM$.
For every $\bfs\in\MMM$, write $\bfs=(\bfx,z)$, where $\bfx\in\PPP$ and $z\in[0,1)$.
The well known Fubini theorem implies that
\begin{displaymath}
\int_\MMM\chi_\WWW(\bfs)\,\dd\bfs
=\int_\PPP\left(\int_{[0,1)}\chi_\WWW(\bfx;z)\,\dd z\right)\dd\lambda_2,
\end{displaymath}
where the inner integral
\textcolor{white}{xxxxxxxxxxxxxxxxxxxxxxxxxxxxxx}
\begin{displaymath}
\psi(\bfx)=\int_{[0,1)}\chi_\WWW(\bfx;z)\,\dd z
\end{displaymath}
is well defined in the sense of Lebesgue for almost every $\bfx\in\PPP$.

Consider now the projection of $\T^*_\bfw$ to the polysquare translation surface~$\PPP$,
resulting in the ergodic $\bfv$-shift $\T^*_\bfv$ on~$\PPP$.
Since the function $\psi(\bfx)$ is $\T^*_\bfv$-invariant, it follows from ergodicity that
it has constant value almost everywhere on~$\PPP$.
Thus it follows from \eqref{eq3.1} that
\textcolor{white}{xxxxxxxxxxxxxxxxxxxxxxxxxxxxxx}
\begin{equation}\label{eq3.2}
\frac{1}{s}\le\psi=\frac{f_\WWW}{s}\le\frac{s-1}{s}
\end{equation}
almost everywhere on~$\PPP$.

Since the invariant subset $\WWW\subset\MMM$ is measurable, it follows that
for every $\eps_1>0$, there exists a finite set of $3$-dimensional
axis-parallel rectangular boxes such that their union $\BBB=\BBB(\WWW;\eps_1)$
has the property that the symmetric difference
\begin{displaymath}
\WWW\bigtriangleup\BBB=(\WWW\setminus\BBB)\cup(\BBB\setminus\WWW)
\end{displaymath}
has measure
\textcolor{white}{xxxxxxxxxxxxxxxxxxxxxxxxxxxxxx}
\begin{equation}\label{eq3.3}
\lambda_3(\WWW\bigtriangleup\BBB)<\eps_1.
\end{equation}

We need the following simple technical result.

Let $\chi_\BBB$ denote the characteristic function of $\BBB$ in~$\MMM$.

\begin{lemma}\label{lem32}
Let $\BBB$ be a finite union of axis-parallel rectangular boxes satisfying \eqref{eq3.3}.
For every $\eps_2>0$, there exists $\eps_3>0$ and a finite set of disjoint
axis-parallel rectangular boxes such that their union $\BBB^*=\BBB^*(\BBB;\eps_2;\eps_3)$ satisfies
the following conditions:

\emph{(i)}
The measure $\lambda_3(\MMM\setminus\BBB^*)<\eps_2$.

\emph{(ii)}
We have $\chi_\BBB(\bfs')=\chi_\BBB(\bfs'')$ if the points $\bfs',\bfs''\in\MMM$ belong to the same axis-parallel
rectangular box in the disjoint union~$\BBB^*$.

\emph{(iii)}
The side lengths of each axis-parallel rectangular box in the disjoint union $\BBB^*$ is greater than~$\eps_3$.

\emph{(iv)}
Let $s$ denote the number of atomic squares of~$\PPP$.
Then
\begin{displaymath}
\lambda_2\left(\left\{\bfx\in\PPP:\lambda_1(\{z\in[0,1):(\bfx,z)\in\BBB^*\})>1-\eps_2^{1/2}\right\}\right)>s-\eps_2^{1/2}.
\end{displaymath}
\end{lemma}

\begin{proof}
We assume, without loss of generality, that every axis-parallel rectangular box in the finite union $\BBB$
lies in the interior of an atomic cube of~$\MMM$, as illustrated in the picture on the left in Figure~1,
which only shows $2$ of the $3$ dimensions.
For each face of each axis-parallel rectangular box in an atomic cube of~$\MMM$,
we extend it to an axis-parallel \textit{special} square face of area $1$ in the same atomic cube,
as illustrated by the dashed lines in the picture on the right in Figure~1.
We repeat this process for every axis-parallel rectangular box in the finite union~$\BBB$.
Suppose that $\NNNN=\NNNN(\BBB)$ denotes the total number of distinct axis-parallel rectangular boxes in the union~$\BBB$.
Then there are at most $2\NNNN$ such distinct special square faces in each of the $3$ axis-parallel directions.
It follows that there exists a number $\eps_4=\eps_4(\BBB)>0$ such that any special square face in an atomic cube
has distance at least $\eps_4$ from any other parallel special square face in the atomic cube
or from any parallel boundary square face of the atomic cube.
We now remove every point in $\MMM$ that lies a distance less than $\eps_5>0$ from some special square face
in the atomic cube that contains that point.
Then the set of such points in $\MMM$ that are removed has measure at most $12\NNNN\eps_5$,
and is represented by the regions shaded in light gray in the picture on the right in Figure~1.

\begin{displaymath}
\begin{array}{c}
\includegraphics[scale=0.8]{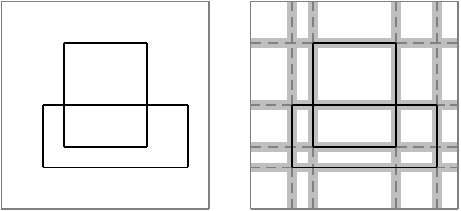}
\\
\mbox{Figure 1: idea behind the construction of the set $\BBB^*$}
\end{array}
\end{displaymath}

Let $\BBB^*$ denote the remainder of $\MMM$ after these points are removed.
Then $\BBB^*$ is clearly a finite union of disjoint axis-parallel rectangular boxes in~$\MMM$,
where each box is either contained in $\BBB$ or disjoint from~$\BBB$, so that the condition (ii) is satisfied.
Suppose now that $\eps_5>0$ is chosen to satisfy
\begin{equation}\label{eq3.4}
12\NNNN\eps_5<\eps_2
\quad\mbox{and}\quad
2\eps_5<\eps_4.
\end{equation}
Then the first condition in \eqref{eq3.4} ensures that the condition (i) is satisfied,
while the second condition in \eqref{eq3.4} ensures that the side lengths of each
axis-parallel rectangular box in the union $\BBB^*$ is greater than $\eps_3=\eps_4-2\eps_5$,
so that the condition (iii) is satisfied.

To establish the condition (iv), note that it follows from the condition (i) and the Fubini theorem that
\textcolor{white}{xxxxxxxxxxxxxxxxxxxxxxxxxxxxxx}
\begin{equation}\label{eq3.5}
\int_\PPP\lambda_1(\{z\in[0,1):(\bfx,z)\not\in\BBB^*\})\,\dd\bfx<\eps_2.
\end{equation}
Let
\textcolor{white}{xxxxxxxxxxxxxxxxxxxxxxxxxxxxxx}
\begin{displaymath}
\EEE=\left\{\bfx\in\PPP:\lambda_1(\{z\in[0,1):(\bfx,z)\not\in\BBB^*\})\ge\eps_2^{1/2}\right\}.
\end{displaymath}
Then it follows from \eqref{eq3.5} that $\lambda_2(\EEE)\eps_2^{1/2}<\eps_2$,
from which the condition (iv) follows immediately, since $\lambda_2(\PPP)=s$.
\end{proof}

For almost every $\bfx\in\PPP$, the set
\begin{displaymath}
U(\WWW;\bfx)=\{z\in[0,1):(\bfx,z)\in\WWW\}
\end{displaymath}
is measurable, and it follows from \eqref{eq3.2} that
\begin{equation}\label{eq3.6}
\frac{1}{s}\le\lambda_1(U(\WWW;\bfx))=\frac{f_\WWW}{s}\le\frac{s-1}{s}.
\end{equation}

We consider Lebesgue measurable subsets $U_\sigma\subset[0,1)$, $\sigma=1,\ldots,s$, of the unit torus $[0,1)$.
In particular, we make the assumption that $0<\lambda_1(U_1)<1$,
where $\lambda_1$ denotes $1$-dimensional Lebesgue measure.
Furthermore, for any real number $u\in\Rr$ and any $\sigma=1,\ldots,s$,
we consider the $(-u)$-translated copy of~$U_\sigma$, given by
\begin{displaymath}
U_\sigma-u=\{\{z-u\}:z\in U_\sigma\}.
\end{displaymath}

Let $\bfx_1,\ldots,\bfx_s\in\PPP$ be distinct points such that their images under projection modulo~$1$
to the unit torus $[0,1)^2$ coincide.
We now apply the following variant of \cite[Lemma~6.1]{BCY-a} to the sets
\begin{equation}\label{eq3.7}
U_1=U(\WWW;\bfx_1),
\quad
\ldots,
\quad
U_s=U(\WWW;\bfx_s),
\end{equation}
so that for every $\sigma=1,\ldots,s$,
\begin{displaymath}
z\in U_\sigma
\quad\mbox{if and only if}\quad
(\bfx_\sigma,z)\in\WWW.
\end{displaymath}

\begin{lemma}\label{lem33}
The set of values $u_0\in[0,1)$ for which the inequalities
\begin{equation}\label{eq3.8}
\lambda_1(U_1\bigtriangleup(U_\sigma-u_0))\ge\frac{1}{32s^2}\lambda_1(U_1)(1-\lambda_1(U_1)),
\quad
\sigma=1,\ldots,s,
\end{equation}
hold simultaneously has Lebesgue measure at least~$1/2$.
\end{lemma}

It then follows on combining \eqref{eq3.6}--\eqref{eq3.8} that
the set of values $u_0\in[0,1)$ for which the inequalities
\textcolor{white}{xxxxxxxxxxxxxxxxxxxxxxxxxxxxxx}
\begin{equation}\label{eq3.9}
\lambda_1(U_1\bigtriangleup(U_\sigma-u_0))\ge\frac{1}{32s^4},
\quad
\sigma=1,\ldots,s,
\end{equation}
hold simultaneously has Lebesgue measure at least~$1/2$.
Let
\begin{equation}\label{eq3.10}
\SSSS(\WWW;\bfx_1)=\{u_0\in[0,1):\mbox{\eqref{eq3.9} holds for every }\sigma=1,\ldots,s\}
\end{equation}
denote this set of such values of~$u_0$.
Note that the condition \eqref{eq3.9} is equivalent to the condition
\begin{equation}\label{eq3.11}
\lambda_1(\{z\in[0,1):\chi_\WWW(\bfx_1,z)\chi_\WWW(\bfx_\sigma,\{z+u_0\})=0\})\ge\frac{1}{32s^4},
\quad
\sigma=1,\ldots,s.
\end{equation}

Let us revisit the disjoint union $\BBB^*=\BBB^*(\BBB;\eps_2;\eps_3)$ of axis-parallel rectangular boxes.
For each axis-parallel rectangular box in this union, we push each boundary face inwards by a distance~$\eps_6$,
where the parameter $\eps_6>0$, to be specified later, satisfies
\textcolor{white}{xxxxxxxxxxxxxxxxxxxxxxxxxxxxxx}
\begin{displaymath}
0<\eps_6<\frac{\eps_3}{2},
\end{displaymath}
where $\eps_3$ is a lower bound of the side lengths of these axis-parallel rectangular boxes.
Then the resulting smaller axis-parallel rectangular box has volume which is at least $(1-2\eps_6/\eps_3)^3$
times that of the original axis-parallel rectangular box.
This means that if $\BBB^{**}=\BBB^{**}(\BBB^*;\eps_6)$ is the disjoint union of these smaller
axis-parallel rectangular boxes, then using Lemma~\ref{lem32}(i) and writing
\begin{equation}\label{eq3.12}
\eps_7=\eps_2+\frac{6s\eps_6}{\eps_3},
\end{equation}
we have
\textcolor{white}{xxxxxxxxxxxxxxxxxxxxxxxxxxxxxx}
\begin{align}\label{eq3.13}
\lambda_3(\BBB^{**})
&
\ge\left(1-\frac{2\eps_6}{\eps_3}\right)^3\lambda_3(\BBB^*)
\ge\left(1-\frac{2\eps_6}{\eps_3}\right)^3(s-\eps_2)
\nonumber
\\
&
\ge\left(1-\frac{6\eps_6}{\eps_3}\right)(s-\eps_2)
\ge s-\eps_2-\frac{6s\eps_6}{\eps_3}
=s-\eps_7.
\end{align}

\begin{remark}
Note that the $\eps_6$-neighbourhood of every point in $\BBB^{**}$ is contained in~$\BBB^*$.
\end{remark}

Observe that $\chi_\BBB(\bfs')=\chi_\BBB(\bfs'')$ if the points $\bfs',\bfs''\in\MMM$ belong to the same axis-parallel
rectangular box in the disjoint union~$\BBB^{**}$,
and that the side lengths of each axis-parallel rectangular box in the disjoint union $\BBB^{**}$ is greater than~$\eps_3-2\eps_6$.
Furthermore, analogous to Lemma~\ref{lem32}(iv), we have the inequality
\begin{equation}\label{eq3.14}
\lambda_2\left(\left\{\bfx\in\PPP:\lambda_1(\{z\in[0,1):(\bfx,z)\in\BBB^{**}\})>1-\eps_7^{1/2}\right\}\right)>s-\eps_7^{1/2}.
\end{equation}

For any $\bfx\in\PPP$, let $\bfx_1,\ldots,\bfx_s\in\PPP$ be distinct points such that their images under projection
modulo~$1$ to the unit torus $[0,1)^2$ coincides with the image of $\bfx$ under the same projection.
Then it follows from \eqref{eq3.14} that
\begin{align}\label{eq3.15}
&
\lambda_2\left(\left\{\bfx\in\PPP:\lambda_1(\{z\in[0,1):(\bfx_\sigma,z)\in\BBB^{**}\})>1-\eps_7^{1/2}
\mbox{ for every $\sigma=1,\ldots,s$}\right\}\right)
\nonumber
\\
&\quad
>s-s\eps_7^{1/2}.
\end{align}

For every $u_0\in[0,1)$, let
\begin{displaymath}
\mu(u_0)=\lambda_2(\{\bfx_1\in\PPP:u_0\in\SSSS(\WWW;\bfx_1)\})
\end{displaymath}
denote some \textit{relevant multiplicity} of~$u_0$.
Then
\begin{displaymath}
\int_{[0,1)}\mu(u_0)\,\dd u_0
=\int_\PPP\lambda_1(\SSSS(\WWW;\bfx_1))\,\dd\bfx_1
\ge\frac{s}{2},
\end{displaymath}
in view of \eqref{eq3.10}.
This can be interpreted to say that the \textit{average multiplicity} $\mu(u_0)$ of $u_0\in[0,1)$ is at least~$s/2$.
Thus there exists a shift $u^*_0\in[0,1)$ such that
\begin{displaymath}
\mu(u_0^*)=\lambda_2(\{\bfx_1\in\PPP:u_0^*\in\SSSS(\WWW;\bfx_1)\})\ge\frac{s}{2}.
\end{displaymath}
It also follows from \eqref{eq3.15} that with the exception of at most $s\eps_7^{1/2}$ part of $\bfx\in\PPP$,
at least $1-2\eps_7^{1/2}$ part of the real numbers $z\in[0,1)$ are such that with $\bfx=\bfx_1$,
\begin{displaymath}
(\bfx_\sigma,z)\in\BBB^{**},
\quad
(\bfx_\sigma,\{z+u_0^*\})\in\BBB^{**},
\quad
\sigma=1,\ldots,s.
\end{displaymath}

In order to derive a contradiction, we need a density analogue of \cite[Lemma~6.2]{BCY-a}.
Here $\Vert\beta\Vert$ denotes the distance of $\beta\in\Rr$ to the nearest integer.

\begin{lemma}\label{lem34}
Let $\bfv=(v_1,v_2)\in\Rr^2$ be a Kronecker vector, and let $w\in\Rr$ be arbitrary such that
$\bfw=(v_1,v_2,w)\in\Rr^3$ is a Kronecker vector.
Let $\eps_6>0$ be given.
There exists a finite sequence
\textcolor{white}{xxxxxxxxxxxxxxxxxxxxxxxxxxxxxx}
\begin{displaymath}
1\le m_1(\eps_6)<\ldots<m_k(\eps_6)
\end{displaymath}
of positive integers such that
\begin{equation}\label{eq3.16}
\Vert m_j(\eps_6)v_1\Vert<\eps_6,
\quad
\Vert m_j(\eps_6)v_2\Vert<\eps_6,
\quad
j=1,\ldots,k,
\end{equation}
and the finite sequence $\{m_j(\eps_6)w\}$, $j=1,\ldots,k$, visits every subinterval of $[0,1)$ with length~$\eps_6$.
\end{lemma}

\begin{remark}
Note that the number $k=k(\bfw;\eps_6)$ of terms of the finite sequence may depend on the choice of $\bfw$
and the value of~$\eps_6$.
\end{remark}

\begin{proof}[Proof of Lemma~\ref{lem34}]
By the Kronecker density theorem, the sequence
\begin{displaymath}
j\bfw=j(v_1,v_2,w),
\quad
j=1,2,3,\ldots,
\end{displaymath}
modulo~$1$ is dense in the unit torus $[0,1)^3$.
Let
\begin{displaymath}
m_1(\eps_6),m_2(\eps_6),m_3(\eps_6),\ldots
\end{displaymath}
be the infinite subsequence of $1,2,3,\ldots$ such that
\begin{displaymath}
\{m_j(\eps_6)v_1\},\{m_j(\eps_6)v_2\}\in[0,\eps_6)\cup(1-\eps_6,1),
\quad
j=1,2,3,\ldots.
\end{displaymath}
Clearly \eqref{eq3.16} holds.
Also, the subsequence $m_j(\eps_6)\bfw$, $j=1,2,3,\ldots,$ modulo~$1$ is dense in $([0,\eps_6)\cup(1-\eps_6,1))^2\times[0,1)\subset[0,1)^3$.
This implies that the sequence $\{m_j(\eps_6)w\}$, $j=1,2,3,\ldots,$ is dense in $[0,1)$.

Next, let the integer $n$ satisfy $n>2/\eps_6$.
We now partition the unit torus $[0,1)$ into $n$ short intervals $I_1,\ldots,I_n$ of length $1/n$ in the standard way.
Then for every $i=1,\ldots,n$, there exists an integer $k_i$ such that the finite sequence
\begin{displaymath}
\{m_j(\eps_6)w\},
\quad
j=1,\ldots,k_i,
\end{displaymath}
visits~$I_i$.
Let $k=\max\{k_1,\ldots,k_n\}$.
Then the finite sequence
\begin{equation}\label{eq3.17}
\{m_j(\eps_6)w\},
\quad
j=1,\ldots,k,
\end{equation}
visits every interval $I_1,\ldots,I_n$.
Now, since $1/n<\eps_6/2$, every subinterval of $[0,1)$ with length $\eps_6$ must contain at least one of the intervals
$I_1,\ldots,I_n$, so is visited by the finite sequence \eqref{eq3.17}.
\end{proof}

It follows that there exists $j_0=j_0(u_0^*)$ satisfying $1\le j_0\le k$ such that
\begin{equation}\label{eq3.18}
\Vert m_{j_0}(\eps_6)v_1\Vert<\eps_6,
\quad
\Vert m_{j_0}(\eps_6)v_2\Vert<\eps_6,
\quad
\Vert m_{j_0}(\eps_6)w-u_0^*\Vert<\eps_6.
\end{equation}

\begin{remark}
For a fixed point $\bfx\in\PPP$, we can visualize the set $\{(\bfx,z):z\in[0,1)\}$ as a \textit{circle}
over the point~$\bfx$, since $[0,1)$ is the unit torus.
For any point $(\bfx,z)$ on this circle, we have $\T^*_\bfw(\bfx,z)=(\bfx+\bfv,\{z+w\})$.
It follows that the image of the circle under the transformation $\T^*_\bfw$ is a circle
$\{(\bfy,z'):z'\in[0,1)\}$ over the point $\bfy=\bfx+\bfv$.
Clearly, this new circle is rotated from the original circle by a quantity~$w$
and its position on $\PPP$ is translated from the original position by a vector~$\bfv$.
This action is repeated multiple times when we apply the transformation $\T^*_\bfw$ successively.
Our proof of Theorem~\ref{thm4}((i)$\Rightarrow$(ii)) is based on a combination
of this fact with Lemmas \ref{lem32}--\ref{lem34}.
\end{remark}

Let us continue the discussion prior to Lemma~\ref{lem34}.
For at least $s/2-s\eps_7^{1/2}$ part of the points $\bfx\in\PPP$, we have $u_0^*\in\SSSS(\WWW;\bfx)$,
and that for at least $1-2\eps_7^{1/2}$ part of the real numbers $z\in[0,1)$, writing $\bfx=\bfx_1$, we have
\begin{equation}\label{eq3.19}
(\bfx,z)\in\BBB^{**},
\quad
(\bfx_\sigma,\{z+u_0^*\})\in\BBB^{**},
\quad
\sigma=1,\ldots,s.
\end{equation}
We say that the points in \eqref{eq3.19} form a \textit{good} $(s+1)$-tuple, in the sense that they all belong to~$\BBB^{**}$.

Let
\textcolor{white}{xxxxxxxxxxxxxxxxxxxxxxxxxxxxxx}
\begin{displaymath}
Q=(\bfx,z)=(\bfx_1,z),
\quad
Q_\sigma=(\bfx_\sigma,\{z+u_0^*\}),
\quad
\sigma=1,\ldots,s,
\end{displaymath}
be such a good $(s+1)$-tuple, and consider the new point $Q^*=(\T^*_\bfw)^{m_{j_0}}(Q)$,
obtained from $Q$ by $m_{j_0}$ successive applications of the transformation~$\T^*_\bfw$,
where $j_0$ is chosen so that the inequalities \eqref{eq3.18} hold.
Since the subset $\WWW\subset\MMM$ is invariant under the transformation~$\T^*_\bfw$,
it follows that $\chi_\WWW(Q)=\chi_\WWW(Q^*)$.

On the other hand, as observed in the Remark above, the image of the circle $\{(\bfx,z):z\in[0,1)\}$
under the transformation $(\T^*_\bfw)^{m_{j_0}}$ is a circle $\{(\bfy,z'):z'\in[0,1)\}$
with some particular $\bfy=\bfx_1+m_{j_0}\bfv\in\PPP$.
It then follows from \eqref{eq3.18} that the coordinates of $\bfy$ are $\eps_6$-close to
the corresponding coordinates of $\bfx_{\sigma_0}$ for a particular $\sigma_0=1,\ldots,s$,
and the last coordinate of $Q^*$ is $\eps_6$-close to $\{z+u_0^*\}$.
Since $Q_{\sigma_0}\in\BBB^{**}$, it follows from the Remark after \eqref{eq3.13} that
$Q_{\sigma_0}$ and $Q^*$ are in the same axis-parallel rectangular box in the disjoint union~$\BBB^*$,
and so $\chi_\BBB(Q_{\sigma_0})=\chi_\BBB(Q^*)$, in view of Lemma~\ref{lem32}(ii).

Recall that $u_0^*\in\SSSS(\WWW;\bfx)$ for at least $s/2-s\eps_7^{1/2}$ part of the points $\bfx\in\PPP$.
This and the condition \eqref{eq3.11} together imply that
\begin{displaymath}
\lambda_1(\{z\in[0,1):\chi_\WWW(\bfx_1,z)\chi_\WWW(\bfx_{\sigma_0},\{z+u_0^*\})=0\})\ge\frac{1}{32s^4},
\end{displaymath}
and this is equivalent to
\begin{equation}\label{eq3.20}
\lambda_1(\{z\in[0,1):\chi_\WWW(Q)\ne\chi_\WWW(Q_{\sigma_0})\})\ge\frac{1}{32s^4}.
\end{equation}
Consider now the three relations
\begin{displaymath}
\chi_\WWW(Q)=\chi_\WWW(Q^*),
\quad
\chi_\BBB(Q_{\sigma_0})=\chi_\BBB(Q^*),
\quad
\chi_\WWW(Q)\ne\chi_\WWW(Q_{\sigma_0}),
\end{displaymath}
which clearly imply the two relations
\begin{displaymath}
\chi_\WWW(Q_{\sigma_0})\ne\chi_\WWW(Q^*),
\quad
\chi_\BBB(Q_{\sigma_0})=\chi_\BBB(Q^*).
\end{displaymath}
It follows that
\textcolor{white}{xxxxxxxxxxxxxxxxxxxxxxxxxxxxxx}
\begin{equation}\label{eq3.21}
\chi_\WWW(Q_{\sigma_0})\ne\chi_\BBB(Q_{\sigma_0})
\quad\mbox{or}\quad
\chi_\WWW(Q^*)\ne\chi_\BBB(Q^*).
\end{equation}
Intuitively speaking, \eqref{eq3.21} represents two \textit{negligible} cases, with total measure less than~$\eps_1$,
in view of \eqref{eq3.3}, which contradict the substantial constant lower bound in \eqref{eq3.20}
if $\eps_1>0$ is sufficiently small.
To make this precise, we need to study more closely the various parameters.

We have $u_0^*\in\SSSS(\WWW;\bfx)$ for at least $s/2-s\eps_7^{1/2}$ part of the points $\bfx\in\PPP$.
Using \eqref{eq3.20}, we deduce that for at least
\textcolor{white}{xxxxxxxxxxxxxxxxxxxxxxxxxxxxxx}
\begin{equation}\label{eq3.22}
\frac{1}{32s^4}-2\eps_7^{1/2}
\end{equation}
part of the real numbers $z\in[0,1)$, the points
\begin{equation}\label{eq3.23}
Q_{\sigma_0}=(\bfx_{\sigma_0},\{z+u_0^*\})
\quad\mbox{and}\quad
Q^*=(\T^*_\bfw)^{m_{j_0}}(Q)=(\T^*_\bfw)^{m_{j_0}}(\bfx_1,z)
\end{equation}
exhibit the property \eqref{eq3.21}.
It follows from \eqref{eq3.22} that the $3$-dimensional Lebesgue measure of the points
$Q=(\bfx_1,z)\in\MMM$ such that  the points \eqref{eq3.23} exhibit the property \eqref{eq3.21} is at least
\textcolor{white}{xxxxxxxxxxxxxxxxxxxxxxxxxxxxxx}
\begin{equation}\label{eq3.24}
\left(\frac{s}{2}-s\eps_7^{1/2}\right)\left(\frac{1}{32s^4}-2\eps_7^{1/2}\right),
\end{equation}
where $\eps_7$ is given by \eqref{eq3.12}.

On the other hand, the property \eqref{eq3.21} is \textit{exceptional},
and \eqref{eq3.3} implies that the quantity in \eqref{eq3.24} is less than~$2\eps_1$.
We emphasize the fact that the choice of the parameter $\eps_1$ is independent of the choices
of the other parameters $\eps_2,\eps_3,\eps_6$.
Thus we can make $\eps_7$ in \eqref{eq3.24} arbitrarily small independently of the fixed value of~$\eps_1$.
It is therefore easy to specify the parameters $\eps_1,\eps_2,\eps_3,\eps_6$ so that the value of the quantity in
\eqref{eq3.24} is greater than~$2\eps_1$, leading to a contradiction.

The contradiction establishes the ergodicity of the $\bfw$-shift in~$\MMM$.

The last step is to extend ergodicity of the $\bfw$-shift in $\MMM$ to unique ergodicity
by using the standard argument in functional analysis.
This is possible, since the projection of $\MMM$ to the unit torus $[0,1)^3$
leads to unique ergodicity there.
\end{proof}

%
%


\begin{thebibliography}{9}

\bibitem{BCY24}
J. Beck, W.W.L. Chen, Y. Yang.
\textit{Non-Integrable Dynamics: Time-Quantitative Results}
(World Scientific, 2024).

\bibitem{BCY-a}
J. Beck, W.W.L. Chen, Y. Yang.
Uniformity of geodesic flow in non-integrable $3$-manifolds
(preprint, 45 pp.).

\bibitem{BDY20a}
J. Beck, M. Donders, Y. Yang.
Quantitative behavior of non-integrable systems I.
\textit{Acta Math. Hungar.} \textbf{161} (2020), 66--184.

\bibitem{gutkin84}
E. Gutkin.
Billiards on almost integrable polyhedral surfaces.
\textit{Ergodic Theory Dynam. Systems} \textbf{4} (1984), 569--584.

\bibitem{veech87}
W.A. Veech.
Boshernitzan's criterion for unique ergodicity of an interval exchange transformation.
\textit{Ergodic Theory Dynam. Systems} \textbf{7} (1987), 149--153.

\end{thebibliography}
\end{document}